\documentclass{amsart}

\usepackage{tikz-cd}
\usepackage[hyphens]{url}

\usepackage{amsthm,amsmath,amssymb,amsfonts,dsfont}
\usepackage[colorlinks, linkcolor=blue, citecolor=blue]{hyperref}
\usepackage[indentafter]{titlesec}
\titleformat{name=\section}{}{\thetitle.}{0.8em}{\centering\scshape}
\usepackage{mathrsfs}
\usepackage{bm}
\usepackage{titlesec}
\numberwithin{equation}{section}
\titleformat{\subsection}[runin]
  {\normalfont\bfseries}{\thesubsection}{1em}{}

\newtheorem{definition}{\textbf{Definition}}[section]
\newtheorem{theorem}[definition]{\textbf{Theorem}}
\newtheorem{corollary}[definition]{\textbf{Corollary}}
\newtheorem{lemma}[definition]{\textbf{Lemma}}

\newtheorem{example}[definition]{\textbf{Example}}
\newtheorem{remark}[definition]{\textbf{Remark}}
\newtheorem{letterthm}{Theorem}

\newtheorem{lettercor}[letterthm]{Corollary}

\def\wx{\infty}

\def\d{\mathrm{d}}

\def\xia{\limits}

\def\car{\curvearrowright}
\def\supp{\mathrm{supp \, }}

\def\prob{\mathrm{Prob}}
\def\BL{B(L^2(M))}

\def\CS{\mathcal{S}}
\def\CB{\mathcal{B}}
\def\CP{\mathcal{P}}
\def\CA{\mathcal{A}}

\def\id {\mathrm{id}}
\def\lag {\langle}
\def\rag {\rangle}

\def\R {\mathbb{R}}
\def\C {\mathbb{C}}
\def\UCP {\mathrm{UCP}}

\def\RN{\mathrm{RN}}

\title[Rigidity of Furstenberg entropy under ucp maps]{Rigidity of Furstenberg entropy under ucp maps}
\author{Shuoxing Zhou}
\address{\'Ecole Normale Sup\'erieure\\ D\'epartement de math\'ematiques et applications\\ 45 rue d'Ulm\\ 75230 Paris Cedex 05\\ FRANCE}
\email{shuoxing.zhou@ens.psl.eu}
\urladdr{\href{https://sites.google.com/view/sxzhou}{\textcolor{black}{https://sites.google.com/view/sxzhou}}}

\begin{document}

\begin{abstract}
Given a tracial von Neumann algebra $(M,\tau)$, we prove that a state preserving $M$-bimodular ucp map between two stationary W$^*$-extensions of $(M,\tau)$ preserves the Furstenberg entropy if and only if it induces an isomorphism between the Radon-Nikodym factors. With a similar proof, we extend this result to quasi-factor maps between stationary spaces of locally compact groups and prove an entropy separation between unique stationary and amenable spaces. As applications, we use these results to establish rigidity phenomena for unique stationary Poisson boundaries.
\end{abstract}
 \maketitle

\section{Introduction}

Let $G$ be a locally compact second countable group and $(X,\nu_X)$ be a nonsingular $G$-space. Following \cite{KV83} and \cite{NZ00}, the Radon-Nikodym factor $(X_\RN,\nu_{X_\RN})$ of $(X,\nu_X)$ is defined to be the minimal $G$-factor that preserves the information of Radon-Nikodym cocycles $\{\frac{\d g\nu_X}{\d\nu_X}\}_{g\in G}$. That is, the completed $\sigma$-algebra $\overline{\Sigma}_{X_\RN}$ of $(X_\RN,\nu_{X_\RN})$ is the minimal $\sigma$-subalgebra of $\overline{\Sigma}_X$ that keeps $\{\frac{\d g\nu_X}{\d\nu_X}\}_{g\in G}$ measurable. 

Let $(M,\tau)$ be a separable tracial von Neumann algebra. As the noncommutative analogue of nonsingular action $G\car(X,\nu_X)$, a W$^*$-inclusion $(M,\tau)\subset(\CA,\varphi_\CA)$ is a von Neumann algebra $\CA$ with $M\subset\CA$ and a normal faithful state $\varphi_\CA$ on $\CA$ with $\varphi_\CA|_M=\tau$. The study of noncommutative ergodic theory in the setting of W$^*$-inclusions was initiated by Das-Peterson \cite{DP20}. Recently, \cite{Zh23} extended the notion of Radon-Nikodym factors to the noncommutative setting (see Subsection \ref{RN factor def} for details).

Both the classical and noncommutative Radon-Nikodym factors exhibit strong connections to the Furstenberg entropy (Subsection \ref{NC entropy}). 
Let $\mu\in\prob(G)$ be an admissible measure and $\varphi\in\CS_\tau(\BL)$ be a normal regular strongly generating hyperstate (Subsection \ref{hyperstate} and \ref{NCPB}). Firstly, the Radon-Nikodym factor always preserves the Furstenberg entropy. Furthermore, for a factor map $\pi:(Y,\nu_Y)\to(X,\nu_X)$ between $(G,\mu)$-spaces or an inclusion map $\id_\CA:(\CA,\varphi_\CA)\hookrightarrow(\CB,\varphi_\CB)$ between 
$\varphi$-stationary W$^*$-extensions, the map preserves the Furstenberg entropy if and only if the two sides of the map share the same Radon-Nikodym factor (\cite[Proposition 1.9]{NZ00} and \cite[Theorem 3.6]{Zh23}). In this paper, we generalize this result to state preserving $M$-bimodular ucp maps $\CP:(\CA,\varphi_\CA)\to(\CB,\varphi_\CB)$ and quasi-factor maps, i.e., $G$-equivariant measurable maps $\beta:Y\to\prob(X)$ with $\nu_X=\int_Y\beta_y\d\nu_Y(y)$, following the terminology from \cite{Gla03}. The main results are as follows:

\begin{letterthm}[Theorem \ref{Thm NC}]\label{Thm A}
Let $(\CA,\varphi_\CA)$ and $(\CB,\varphi_\CB)$ be $\varphi$-stationary W$^*$-extensions of $(M,\tau)$. Assume that there exists a state preserving $M$-bimodular ucp map $\CP:(\CA,\varphi_\CA)\to (\CB,\varphi_\CB)$. Then 
$$h_\varphi(\CA,\varphi_\CA)\leq h_\varphi(\CB,\varphi_\CB).$$
Assume that $h_\varphi(\CB,\varphi_\CB)<+\wx$. Then the equality holds if and only if $\CP|_{\CA_\mathrm{RN}}:\CA_\mathrm{RN}\to\CB_\mathrm{RN}$ is a $\ast$-isomorphism between the Radon-Nikodym factors of $(\CA,\varphi_\CA)$ and $(\CB,\varphi_\CB)$.
\end{letterthm}

Following an operator algebraic proof in the same spirit as that of Theorem \ref{Thm A}, we have the following main result for locally compact groups. This is an unusual example where a general theorem (Theorem \ref{Thm A}) in the noncommutative ergodic theory of W$^*$-inclusions has preceded and inspired its counterpart (Theorem \ref{Thm B}) in ergodic group theory, while the usual case is the converse.

\begin{letterthm}[Theorem \ref{P for G} and \ref{beta for G}]\label{Thm B}
Let $(X,\nu_X)$ and $(Y,\nu_Y)$ be compact metrizable $(G,\mu)$-spaces. Assume that there exists a quasi-factor map, i.e., a $G$-equivariant measurable map $\beta:Y\to \prob(X)$ with $\nu_X=\int_Y\beta_y\d\nu_Y(y)$. Then we have
$$h_\mu(X,\nu_X)\leq h_\mu(Y,\nu_Y).$$
Assume that $h_\mu(Y,\nu_Y)<+\wx$. Then the equality holds if and only if $\beta$ induces a measure preserving $G$-equivariant isomorphism $\beta_\mathrm{RN}:(Y_\mathrm{RN},\nu_{Y_\mathrm{RN}})\to(X_\mathrm{RN},\nu_{X_\mathrm{RN}})$ between the Radon-Nikodym factors with $\pi_{X_{\mathrm{RN}}*}\circ\beta(y)=\delta_{\beta_\mathrm{RN}\circ\pi_{Y_\mathrm{RN}}(y)}$ for $\nu_Y$-a.e. $y\in Y$, i.e., the following commutative diagram:
$$\begin{tikzcd}
Y \arrow[d, "\pi_{Y_\mathrm{RN}}", two heads] \arrow[r, "\beta"] & \prob(X) \arrow[d, "\pi_{X_\mathrm{RN}*}", two heads] \\
Y_\mathrm{RN} \arrow[r, "\beta_\mathrm{RN}"] \arrow[r, "\sim"']  & X_\mathrm{RN}\subset \prob(X_\mathrm{RN}).            
\end{tikzcd}$$
\end{letterthm}

As a corollary of Theorem \ref{Thm B}, we prove that there is an entropy separation between unique stationary $(G,\mu)$-spaces and amenable $(G,\mu)$-spaces, as well as the noncommutative analogue (Corollary \ref{us leq am for M}). We refer to Definition \ref{unique stationary} for the definition of the unique stationary property and Definition \ref{def RN irr} for RN-irreducibility.
\begin{lettercor}[Corollary \ref{us leq am for G}]\label{Cor C}
Let $(X,\nu_X)$ be a $\mu$-unique stationary $(G,\mu)$-space and $(Y,\nu_Y)$ be an amenable $(G,\mu)$-space. Then we have
$$h_\mu(X,\nu_X)\leq h_\mu(Y,\nu_Y).$$
In particular, let
$$h_\mathrm{us}(\mu)=\sup\{h_\mu(X,\nu_X)\mid\mbox{$(X,\nu_X)$ is a $\mu$-unique stationary $(G,\mu)$-space}\},$$
$$h_\mathrm{am}(\mu)=\inf\{h_\mu(Y,\nu_Y)\mid\mbox{$(Y,\nu_Y)$ is an amenable $(G,\mu)$-space}\}.$$
Then we have
$$h_\mathrm{us}(\mu)\leq h_\mathrm{am}(\mu).$$
Moreover, assume that $h_\mu(B,\nu_B)<+\wx$, then up to isomorphisms, $G$ admits at most one RN-irreducible $(G,\mu)$-space that is $\mu$-unique stationary and amenable.
\end{lettercor}
As an application of Corollary \ref{Cor C}, we give new proofs to some of the main results in \cite{HK24} (Corollary \ref{HK24 Thm A} and \ref{HK24 Thm B}), which study amenable intermediate factors and subalgebras under the assumption of unique stationary Poisson boundaries.

By applying Theorem \ref{Thm B}, we also provide new proofs to some classical results like \cite[Theorem 4.4]{FG10} and \cite[Theorem 9.2]{NS13} (Corollary \ref{max entropy for G} and  \ref{USB 3 eq for G}), which characterize spaces with maximal entropy. Moreover, by applying Theorem \ref{Thm A}, we are able to generalize these results to the noncommutative setting:

Let $(\CB_\varphi,\zeta)$ be the $\varphi$-Poisson boundary (Subsection \ref{NCPB}). \begin{lettercor}[Corollary \ref{max entropy for M}]
Assume that $h_\varphi(\CB_\varphi,\zeta)<+\wx$. Let $(\CA,\varphi_\CA)$ be a $\varphi$-stationary W$^*$-extension of $(M,\tau)$. Then $h_\varphi(\CA,\varphi_A)=h_\varphi(\CB_\varphi,\zeta)$ if and only if $(\CB_\varphi,\zeta)\subset(\CA,\varphi_\CA)$ as the Radon-Nikodym factor.
\end{lettercor}

We say that a nonsingular space or W$^*$-extension is RN-irreducible if its Radon-Nikodym factor is identified with itself (Definition \ref{def RN irr}). 
\begin{lettercor}[Corollary \ref{USB 3 eq for M}]
Assume that $(\CB_\varphi,\zeta)$ is $\varphi$-unique stationary and $h_\varphi(\CB_\varphi,\zeta)<+\wx$. Assume that $(\CA,\varphi_\CA)$ is a RN-irreducible $\varphi$-stationary W$^*$-extension of $(M,\tau)$. Then the following conditions are equivalent: 
\begin{itemize}
    \item [(i)] $h_\varphi(\CA,\varphi_\CA)=h_\varphi(\CB_\varphi,\zeta)$;
    \item [(ii)] $(\CA,\varphi_\CA)\cong(\CB_\varphi,\zeta)$ as W$^*$-extensions;
    \item[(iii)] $\CA$ is amenable. 
\end{itemize}
In particular, such an inclusion $M\subset \CB_\varphi$ does not admit any non-trivial amenable intermediate subalgebra.
\end{lettercor}

\section{Preliminaries}\label{Preliminaries}

\subsection{Hyperstates and ucp maps.}\label{hyperstate}
Fix a separable tracial von Neumann algebra $(M, \tau )$. Following \cite{DP20}, for a C$^*$-algebra $\CA$ with $M\subset\CA$ and a state $\varphi_\CA$ on $\CA$, we say that $\varphi_\CA$ is a $\tau$-\textbf{hyperstate} if ${\varphi_\CA}|_M =\tau$. We denote by $\mathcal{S}_\tau(\CA)$ the set of $\tau$-hyperstates on $\CA$. Let $e_M\in B(L^2(\CA,\varphi_\CA))$ be the orthogonal projection onto $L^2(M)$. The ucp map $\CP_{\varphi_\CA}:\CA\to \BL$ is defined as 
$$\CP_{\varphi_\CA}(T)=e_MTe_M, \ T\in \CA.$$
According to \cite[Proposition 2.1]{DP20}, ${\varphi_\CA}\mapsto\CP_{\varphi_\CA}$ is a bijection between (normal) hyperstates on $\CA$ and (normal) $M$-bimodular ucp maps from $\CA$ to $\BL$, whose inverse is $\CP \mapsto \langle\CP(\,\cdot\,)\hat{1},\hat{1}\rangle$. 

For ${\varphi_\CA}\in\mathcal{S}_\tau(\CA)$ and $\varphi \in \mathcal{S}_\tau(B(L^2(M)))$, the \textbf{convolution} $\varphi \ast{\varphi_\CA}\in \mathcal{S}_\tau(\CA)$ is defined to be the hyperstate associated to the $M$-bimodular ucp map $\mathcal{P}_{\varphi}\circ \mathcal{P}_{\varphi_\CA}$. And ${\varphi_\CA}$ is said to be \textbf{$\varphi$-stationary} if $\varphi \ast{\varphi_\CA}=\varphi_\CA$.

For convenience, by a \textbf{W$^*$-inclusion} $(M,\tau)\subset(\CA,\varphi_\CA)$ or a \textbf{W$^*$-extension} $(\CA,\varphi_\CA)$ of $(M,\tau)$, we mean that $M\subset\CA$ is a W$^*$-inclusion and $\varphi_\CA\in\CS_\tau(\CA)$ is a normal faithful hyperstate.

\subsection{Noncommutative Poisson boundaries.}\label{NCPB}
  Let $\varphi\in\mathcal{S}_\tau(B(L^2(M)))$ be a hyperstate. Following \cite{DP20}, the set of $\mathcal{P}_\varphi$\textbf{-harmonic operators} is defined to be
    $$\mathrm{Har}(\mathcal{P}_\varphi)=\mathrm{Har}(B(L^2(M)),\mathcal{P}_\varphi)=\{ T\in B(L^2(M))\mid \mathcal{P}_\varphi(T)=T\}. $$
    The \textbf{noncommutative Poisson boundary} $\mathcal{B}_\varphi$ of $M$ with respect to $\varphi$ is defined to be the noncommutative Poisson boundary of the ucp map $\mathcal{P}_\varphi$ as
defined by Izumi \cite{Izu02}, that is, the Poisson boundary $\mathcal{B}_\varphi$ is the unique $C^*$-algebra (a von Neumann algebra when $\varphi$ is normal) that is isomorphic, as an operator system, to the space of harmonic operators $\mathrm{Har}(\mathcal{P}_\varphi)$. And the isomorphism $\CP:\CB_\varphi\to \mathrm{Har}(\mathcal{P}_\varphi)$ is called the $\varphi$-\textbf{Poisson transform}. Since $M\subset\mathrm{Har}(\mathcal{P}_\varphi)$, $M$ can also be embedded into $\CB_\varphi$ as a subalgebra. 

According to \cite[Proposition 2.8]{DP20}, for a normal hyperstate $\varphi\in\mathcal{S}_\tau(B(L^2(M)))$, there exists a sequence $\{z_n\}\subset M$ such that $\sum_{n=1}^{\wx}z_n^*z_n=1$, and $\varphi$ and $\CP_\varphi$ admit the following standard form:
$$\varphi(T)=\sum_{n=1}^{\wx}\langle T  \hat{z}_n,\hat{z}_n
     \rangle, \ \mathcal{P}_\varphi(T)=\sum_{n=1}^{\wx} (Jz_nJ)T(Jz_n^*J), \ T\in \BL.$$
Following \cite{DP20}, $\varphi$ is said to be 
   \begin{itemize}
        \item \textbf{regular}, if $\sum_{n=1}^{\wx}\xia z_n^*z_n=\sum_{n=1}^{\wx}\xia z_nz_n^*=1$;
        \item \textbf{strongly generating}, if the unital algebra (rather than the unital $\ast$-algebra) generated by $\{z_n\}$ is weakly dense in $M$.
    \end{itemize}
According to \cite{DP20}, when $\varphi$ is a normal regular strongly generating hyperstate, $\zeta:=\varphi\circ\CP\in\mathcal{S}_\tau(\CB_\varphi)$ is a normal faithful hyperstate on $\CB_\varphi$.

\subsection{Entropy.}\label{NC entropy} Following \cite{DP20}, for a normal regular strongly generating hyperstate $\varphi\in\mathcal{S}_\tau(B(L^2(M)))$, let $A_\varphi\in \BL$ be the trace class operator associated with $\varphi$. The \textbf{entropy} of $\varphi$ is defined to be 
$$H(\varphi)=-\mathrm{Tr}(A_\varphi \log A_\varphi).$$
The \textbf{asymptotic entropy} of $\varphi$ is defined to be 
$$h(\varphi)=\lim_{n\to\wx}\frac{H(\varphi^{*n})}{n}.$$

For a W$^*$-inclusion $(M,\tau)\subset(\CA,\varphi_\CA)$, let $\Delta_\CA:L^2(\CA,\varphi_\CA)\to L^2(\CA,\varphi_\CA)$ be the modular operator of $(\CA,\varphi_\CA)$ and $e\in B(L^2(\CA,\varphi_\CA))$ be the orthogonal projection onto $L^2(M)$. The \textbf{Furstenberg entropy} of $(\CA,\varphi_\CA)$ with respect to $\varphi$ is defined to be 
$$h_\varphi(\CA,\varphi_\CA)=-\varphi(e\log\Delta_\CA e).$$

It is shown in \cite{DP20} and \cite{Zh23} the following fundamental theorems regarding noncommutative entropy: Assume that $H(\varphi)<+\wx$, then we have
\begin{itemize}
    \item[(1)] $h_\varphi(\CA,\varphi_\CA)\leq h(\varphi)$ for any $\varphi$-stationary W$^*$-extension $(\CA,\varphi_\CA)$;
    \item[(2)] $h_\varphi(\CB_\varphi,\zeta)=h(\varphi)$;
    \item[(3)] For any $\varphi$-boundary, i.e., intermediate subalgebra $(M,\tau)\subset(\CB_0,\zeta_0)\subset(\CB_\varphi,\zeta)$, one has $h_\varphi(\CB_0,\zeta_0)=h_\varphi(\CB_\varphi,\zeta)$ iff $(\CB_0,\zeta_0)=(\CB_\varphi,\zeta)$;
    \item[(4)] The $\varphi$-Poisson boundary is trivial (i.e., $\CB_\varphi=M$) iff $h(\varphi)=0$.
\end{itemize}

\subsection{Radon-Nikodym factors.}\label{RN factor def}
The notion of Radon-Nikodym factors is firstly introduced for discrete groups in \cite{KV83} and then extended to locally compact groups in \cite{NZ00}. Let $G$ be a locally compact second countable group and $(X,\Sigma_X,\nu_X)$ be a nonsingular $G$-space, where $\Sigma_X$ is the the $\sigma$-algebra of $(X,\nu_X)$ and $\overline{\Sigma}_X$ is the completion with respect to $\nu_X$. Following \cite[Definition 1.12]{NZ00}, the Radon-Nikodym factor $(X_\RN,\Sigma_{X_\RN},\nu_{X_\RN})$ is the unique $G$-factor of $(X,\Sigma_X,\nu_X)$ such that $\overline{\Sigma}_{X_\RN}\subset\overline{\Sigma}_X$ is the minimal $\sigma$-subalgebra that keeps the function family $\{\frac{\d g\nu_X}{\d\nu_X}\}_{g\in G}$ measurable, or equivalently, $L^\wx(X_\RN)\subset L^\wx(X)$ is the von Neumann subalgebra generated by $\{(\frac{\d g\nu_X}{\d\nu_X})^{it}\}_{g\in G,t\in\R}$. And we denote the factor map by $\pi_{X_\RN}:(X,\nu_X)\to(X_\RN,\nu_{X_\RN})$.

Recently, \cite{Zh23} extended the notion of Radon-Nikodym factors to the setting of W$^*$-inclusions. For a W$^*$-inclusion $(M,\tau)\subset(\CA,\varphi_\CA)$, let $\{\sigma_t^\CA\}_{t\in\R}$ be the automorphism group of $(\CA,\varphi_\CA)$. Then following \cite[Definition 3.4]{Zh23}, the \textbf{noncommutative Radon-Nikodym factor} $(\CA,\varphi_\CA)_\RN=(\CA_\RN,\varphi_{\CA_\RN})$ of $(\CA,\varphi_\CA)$ is the von Neumann subalgebra $\CA_\RN\subset\CA$ generated by $\{\sigma_t^\CA(z)\mid z\in M,t\in \R\}$ with the hyperstate $\varphi_{\CA_\RN}=\varphi_\CA|_{\CA_\RN}$.

It is shown in \cite[Lemma 1.13]{NZ00} and \cite[Theorem 3.6]{Zh23} that both classical and noncommutative Radon-Nikodym factors preserve the Furstenberg entropy.

Assume that $G$ is discrete. It is shown in \cite[Example 3.5]{Zh23} that the noncommutative Radon-Nikodym factor of $L(G\car (X,\nu_X))$ with respect to $L(G)$ is exactly $L(G\car(X_\mathrm{RN},\nu_{X_\mathrm{RN}}))$.

\section{Proofs of main results}
The following lemma generalizes \cite[Proposition 1]{Pet86}.
\begin{lemma}\label{(D r+s)t}
For $i=1,2$, let $\Delta_i$ be a unbounded positive operator on Hilbert space $H_i$. If a bounded operator $T:H_1\to H_2$ with $\Vert T\Vert=1$ satisfies 
$$T^*\Delta_2 T\leq\Delta_1,$$
then for any $0\leq r,t\leq 1$ and $s\geq 0$, we have
$$T^*(\Delta_2^t+s)^r T\leq (\Delta_1^t+s)^r. $$
\end{lemma}
\begin{proof}
By \cite[Proposition 1]{Pet86}, for any $0\leq r\leq1$, we have 
\begin{equation}\label{P Dt P Dt}
T^*\Delta_2^r T\leq \Delta_1^r
\end{equation}
Since $T^*T\leq 1$, for any $0\leq t\leq1$, we have 
$$T^*(\Delta_2^t+s)T \leq T^*\Delta_2^t T+s\cdot 1\leq \Delta_1^t+s.$$
By replacing $\Delta_i$ with $\Delta_i^t+s$ $(i=1,2)$ in (\ref{P Dt P Dt}), we have
$$T^*(\Delta_2^t+s)^r T\leq (\Delta_1^t+s)^r.$$
\end{proof}

We fix a separable tracial von Neumann algebra $(M,\tau)$ and a normal regular strongly generating hyperstate $\varphi=\sum_{k=1}^{\wx}\lag\,\cdot\, \hat{z}_k,\hat{z}_k\rag\in \CS_\tau(\BL)$. 
\begin{theorem}\label{Thm NC}
Let $(\CA,\varphi_\CA)$ and $(\CB,\varphi_\CB)$ be $\varphi$-stationary W$^*$-extensions of $(M,\tau)$. Assume that there exists a state preserving $M$-bimodular ucp map $\CP:(\CA,\varphi_\CA)\to (\CB,\varphi_\CB)$. Then 
$$h_\varphi(\CA,\varphi_\CA)\leq h_\varphi(\CB,\varphi_\CB).$$
Assume that $h_\varphi(\CB,\varphi_\CB)<+\wx$. Then the equality holds if and only if $\CP|_{\CA_\mathrm{RN}}:\CA_\mathrm{RN}\to\CB_\mathrm{RN}$ is a $\ast$-isomorphism between the Radon-Nikodym factors of $(\CA,\varphi_\CA)$ and $(\CB,\varphi_\CB)$.
\end{theorem}
\begin{proof}
Let $\bar{\varphi}=\sum_{k=0}^{\wx}2^{-k-1}\varphi^{*k}$. Then by \cite[Corollary 5.11]{DP20}, for any $\varphi$-stationary W$^*$-extension $(\mathcal{C},\varphi_\mathcal{C})$ of $(M,\tau)$, we have
$$h_{\bar{\varphi}}(\mathcal{C},\varphi_\mathcal{C})=\sum_{k=0}^{\wx}2^{-k-1}k\cdot h_\varphi(\mathcal{C},\varphi_\mathcal{C}):= C\cdot h_\varphi(\mathcal{C},\varphi_\mathcal{C}).$$ 
Hence after replacing $\varphi$ with $\bar{\varphi}$, we may assume that $\mathrm{span}\{z_k\hat{1}\}$ is dense in $L^2(M)$. 

Let $P:L^2(\CA,\varphi_\CA)\to L^2(\CB,\varphi_\CB)$ be the bounded operator extended from $P(a\xi_\CA)=\CP(a)\xi_\CB$ ($a\in \CA$), where $\xi_A$ and $\xi_\CB$ are the cyclic vectors. Then we have $\Vert P \Vert=1$, $P(z\xi_\CA)=z\xi_\CB$ and $P^*(z\xi_\CB)=z\xi_\CA$ for any $z\in M$. 

Let $S_\CA:a\xi_\CA\mapsto a^*\xi_\CA$ and $S_\CB:b\xi_\CB\mapsto b^*\xi_\CB$ for $a\in\CA$ and $b\in\CB$. Then we have $PS_\CA=S_\CB P$. Hence the modular operators $\Delta_\CA$ and $\Delta_\CB$ satisfy that 
$$\Delta_\CA=S_\CA^*S_\CA\geq S_\CA^*P^*PS_\CA=P^*S_\CB^*S_\CB P=P^*\Delta_\CB P.$$
By Lemma \ref{(D r+s)t}, for $0\leq t\leq 1$ and $\epsilon>0$, we have 
\begin{equation}\label{PDBP leq DA}
P^*(\Delta_\CB+\epsilon)^t P\leq (\Delta_\CA+\epsilon)^t.
\end{equation}

For $\Delta=\Delta_\CA$ or $\Delta_\CB$, since $\varphi_\CA|_M=\varphi_\CB|_M=\tau$, we know that $\Delta^{1/2}|_{L^2(M)}$ is isometric. Hence for any $z\in M$, we have $z\xi\in D(\Delta^{1/2})\subset D((\Delta+\epsilon)^{t/2})$ ($\xi=\xi_\CA$ or $\xi_\CB$). By (\ref{PDBP leq DA}), we have
$$\lag (\Delta_\CB+\epsilon)^t z\xi_\CB,z\xi_\CB\rag=\lag P^*\Delta_\CB^t P z\xi_\CA,z\xi_\CA\rag\leq \lag (\Delta_\CA+\epsilon)^t  z\xi_\CA,z\xi_\CA\rag.$$
Hence for $t>0$, we have
$$\frac{\lag ((\Delta_\CB+\epsilon)^t-1) z\xi_\CB,z\xi_\CB\rag}{t}\leq\frac{\lag ((\Delta_\CA+\epsilon)^t-1)  z\xi_\CA,z\xi_\CA\rag}{t}.$$
Note that $\log \epsilon\leq\log(\Delta+\epsilon)\leq \Delta+\epsilon$, hence $z\xi\in D(|\log(\Delta+\epsilon)|^{1/2})$. Let $t\to 0$, we have
$$\lag \log(\Delta_\CB+\epsilon) z\xi_\CB,z\xi_\CB\rag\leq\lag \log(\Delta_\CA+\epsilon) z\xi_\CA,z\xi_\CA\rag,$$
which is equivalent to
$$e_M^\CB \log(\Delta_\CB+\epsilon) e_M^\CB\leq e_M^\CA \log(\Delta_\CA+\epsilon) e_M^\CA,$$
where $e_M^\CA: L^2(\CA)\to L^2(M)$ and $e_M^\CB: L^2(\CB)\to L^2(M)$ are the orthogonal  projections. Let $\epsilon\to 0$, we have
\begin{equation}\label{DB leq DA}
e_M^\CB \log\Delta_\CB e_M^\CB\leq e_M^\CA \log\Delta_\CA e_M^\CA.
\end{equation}
Hence
\begin{equation}\label{hB-hA}
h_\varphi(\CA,\varphi_\CA)=-\varphi( e_M^\CA \log \Delta_\CA e_M^\CA)\leq-\varphi(e_M^\CB \log\Delta_\CB e_M^\CB)=h_\varphi(\CB,\varphi_\CB).
\end{equation}

From now on, assume that $h_\varphi(\CB,\varphi_\CB)<+\wx$. 

Assume that $\CP|_{\CA_\mathrm{RN}}:\CA_\mathrm{RN}\to\CB_\mathrm{RN}$ is a $\ast$-isomorphism. Then by \cite[Theorem 3.6]{Zh23}, we have
$$h_\varphi(\CA,\varphi_\CA)=h_\varphi(\CA_\mathrm{RN},\varphi_{\CA_\mathrm{RN}})=h_\varphi(\CB_\mathrm{RN},\varphi_{\CB_\mathrm{RN}})=h_\varphi(\CB,\varphi_\CB).$$

Conversely, assume that $h_\varphi(\CA,\varphi_\CA)=h_\varphi(\CB,\varphi_\CB)<+\wx$. First, we want to prove that $\CP(\sigma^\CA_t(z))=\sigma^\CB_t(z)$ for any $z\in M$ and $t\in \R$, where $\{\sigma_t^\CA\}_{t\in\R}$ and $\{\sigma_t^\CB\}_{t\in\R}$ are the automorphism groups of $(\CA,\varphi_\CA)$ and $(\CB,\varphi_\CB)$.

Fix $k\geq 1$ and consider $z_k$ in the standard form of $\varphi$. Note that for $(\Delta,\xi)=(\Delta_\CA,\xi_\CA)$ or $(\Delta_\CB,\xi_B)$, by $M\xi\subset D(\Delta^{1/2})$ and $$\lag|\log \Delta| z_k\xi,z_k\xi\rag\leq\lag(2\Delta-\log \Delta) z_k\xi,z_k\xi\rag= 2\Vert z_k\Vert_{\tau,2}+\lag\log \Delta z_k\xi,z_k\xi\rag<+\wx$$ (note that $2x\geq|\log x|+\log x$), we have $z_k\xi\in D(\Delta^{1/2})\cap D(|\log \Delta|^{1/2})$. Assume that a Borel function $f:[0,+\wx)\to\C$ satisfies that there exists a constant $C>0$ such that for any $x\geq0$,
\begin{equation}\label{f leq C}
|f(x)|\leq C(1+x^{1/2}+|\log x|^{1/2}).
\end{equation}
Then we have $z_k\xi\in D(f(\Delta))$. Note that every $f(\Delta)(z_k\xi)$ or $\lag |f(\Delta)|^2 z_k\xi,z_k\xi\rag$ appearing in the following proof satisfies the condition (\ref{f leq C}), hence well defined.

Moreover, assume that Borel functions $g$ and $G:[0,
+\wx)\times[a,b]\to \C$ satisfy that for any $x>0$ and $a\leq t\leq b$, $\frac{\d}{\d t}g(x,t)$ and $\frac{\d}{\d t}G(x,t)$ exist, and there exists a constant $C>0$ such that 
$$|g(x,t)|+|\frac{\d}{\d t}g(x,t)|\leq C(1+x^{1/2}+|\log x|^{1/2}),$$ 
$$|G(x,t)|+|\frac{\d}{\d t}G(x,t)|\leq C(1+x^{1/2}+|\log x|^{1/2})^2.$$
Let
$$\xi_1=(1+\Delta^{1/2}+|\log \Delta |^{1/2})z_k\xi,$$
$$g_1(x,t)=g(x,t)(1+x^{1/2}+|\log x|^{1/2})^{-1},$$
$$G_1(x,t)=G(x,t)(1+x^{1/2}+|\log x|^{1/2})^{-2}.$$
Then $\xi_1$ is well defined and all $g_1$, $\frac{\d}{\d t}g_1$, $G_1$ and $\frac{\d}{\d t}G_1$ are bounded by $C$. By writing 
$$g(\Delta,t) z_k\xi=g_1(\Delta,t) \xi_1,$$
$$\lag G(\Delta,t)z_k\xi,z_k\xi\rag=\lag G_1(\Delta,t)\xi_1,\xi_1\rag,$$
we know that $g(\Delta,t) z_k\xi$ and $\lag G(\Delta,t)z_k\xi,z_k\xi\rag$ are differentiable with respect to $t$ and
\begin{equation}\label{dg}
\frac{\d}{\d t}(g(\Delta,t) z_k\xi)=\frac{\d}{\d t}g(\Delta,t) z_k\xi;
\end{equation}
\begin{equation}\label{dG}
\frac{\d}{\d t}\lag G(\Delta,t)z_k\xi,z_k\xi\rag=\left\lag \frac{\d}{\d t}G(\Delta,t) z_k\xi,z_k\xi\right\rag.
\end{equation}

By (\ref{hB-hA}) and (\ref{DB leq DA}) we must have
\begin{equation}\label{log DB=log DA}
\lag \log\Delta_\CA z_k\xi_\CA,z_k\xi_\CA\rag = \lag\log \Delta_\CB z_k\xi_\CB,z_k\xi_\CB\rag.
\end{equation}

Note that $PS_\CA=S_\CB P$ is equivalent to $S_\CA^*P^*=P^*S_\CB^*$. Then we have
$$\Delta_\CB^{-1}=S_\CB S_\CB^*\geq S_\CB PP^* S_\CB^*=PS_\CA S_\CA^*P^*=P\Delta_\CA^{-1}P^*.$$
Take $(\Delta_1,\Delta_2,T)=(\Delta_\CB^{-1},\Delta_\CA^{-1},P^*)$ in Lemma \ref{(D r+s)t}, then for any $1\leq r,t\leq 1$ and $s\geq 0$, we have 
\begin{equation}\label{D-t+s}
P(\Delta_\CA^{-t}+s)^rP^*\leq (\Delta_\CB^{-t}+s)^r.
\end{equation}
To satisfy the condition (\ref{f leq C}), we rewrite (\ref{D-t+s}) as 
\begin{equation}\label{1+sDt}
\Delta_\CB^{\frac{rt}{2}}P(\Delta_\CA^{-t}+s)^rP^*\Delta_\CB^{\frac{rt}{2}}\leq (1+s\Delta_\CB^t)^r.  
\end{equation}
For $\epsilon>0$, we have that $((\Delta_\CA+\epsilon)^{-t}+s)^r$ is bounded and $$((\Delta_\CA+\epsilon)^{-t}+s)^r\leq(\Delta_\CA^{-t}+s)^r.$$
Therefore, by (\ref{1+sDt}), we have
$$\Delta_\CB^{\frac{rt}{2}}P((\Delta_\CA+\epsilon)^{-t}+s)^rP^*\Delta_\CB^{\frac{rt}{2}}\leq (1+s\Delta_\CB^t)^r.$$
For any $k\geq 1$, since $PP^*(z_k\xi_\CB)=z_k\xi_\CB$, we have
$$\frac{\lag (\Delta_\CB^{\frac{rt}{2}}P((\Delta_\CA+\epsilon)^{-t}+s)^rP^*\Delta_\CB^{\frac{rt}{2}}-PP^*)z_k\xi_\CB,z_k\xi_\CB\rag}{r}
\leq\frac{\lag ((1+s\Delta_\CB^t)^r-1)z_k\xi_\CB,z_k\xi_\CB\rag}{r}.$$
Let $r\to0$, by (\ref{dg}), we have
\begin{align*}
&\lag \log \Delta_\CB^{\frac{t}{2}}PP^*z_k\xi_\CB,z_k\xi_\CB\rag\\
+&\lag P\log((\Delta_\CA+\epsilon)^{-t}+s)P^*z_k\xi_\CB,z_k\xi_\CB\rag\\
+&\lag PP^*\log \Delta_\CB^{\frac{t}{2}}z_k\xi_\CB,z_k\xi_\CB\rag\\
\leq &\lag\log (1+s\Delta_\CB^t)z_k\xi_\CB,z_k\xi_\CB\rag.
\end{align*}
By $PP^*(z_k\xi_\CB)=z_k\xi_\CB$ and $P^*(z_k\xi_\CB)=z_k\xi_\CA$, we have
$$\lag \log((\Delta_\CA+\epsilon)^{-t}+s)z_k\xi_\CA,z_k\xi_\CA\rag\leq \lag\log (\Delta_\CB^{-t}+s)z_k\xi_\CB,z_k\xi_\CB\rag.$$
Let $\epsilon\to 0$, we have
\begin{equation}\label{log D-t+s}
\lag \log(\Delta_\CA^{-t}+s)z_k\xi_\CA,z_k\xi_\CA\rag\leq \lag\log (\Delta_\CB^{-t}+s)z_k\xi_\CB,z_k\xi_\CB\rag.
\end{equation}
By (\ref{log DB=log DA}) and $\log(\Delta^{-t})=-t\log \Delta$, we also have
\begin{equation}\label{log Dt=log Dt}
\lag \log(\Delta_\CA^{-t})z_k\xi_\CA,z_k\xi_\CA\rag = \lag\log (\Delta_\CB^{-t})z_k\xi_\CB,z_k\xi_\CB\rag.
\end{equation}
By (\ref{log D-t+s}) and (\ref{log Dt=log Dt}), we have
$$\frac{\lag (\log(\Delta_\CA^{-t}+s)-\log (\Delta_\CA^{-t}))z_k\xi_\CA,z_k\xi_\CA\rag}{s}\leq \frac{\lag (\log(\Delta_\CB^{-t}+s)-\log (\Delta_\CB^{-t}))z_k\xi_\CB,z_k\xi_\CB\rag}{s}.$$
Let $s\to 0$, then by (\ref{dG}), we have
$$\lag \Delta_\CA^t z_k\xi_\CA,z_k\xi_\CA\rag\leq \lag \Delta_\CB^t z_k\xi_\CB,z_k\xi_\CB\rag.$$
That is,
$$\lag (\Delta_\CA^t-P^*\Delta_\CB^tP) z_k\xi_\CA,z_k\xi_\CA\rag\leq 0.$$
But by (\ref{PDBP leq DA}), $\Delta_\CA^t-P^*\Delta_\CB^tP\geq 0$. Hence we must have 
\begin{equation}\label{Dtz_k=Dtz_k}
\Delta_\CA^t z_k\xi_\CA=P^*\Delta_\CB^t P z_k\xi_\CA=P^*\Delta_\CB^t  z_k\xi_\CB
\end{equation}
for any $0<t<1/2$. 

For any $n\geq 0$ and $0<\epsilon<t<1/2-\epsilon$, note that $|\frac{\d^n x^t}{\d t^n}|=|\log x|^n x^t\leq C_{n,\epsilon}(1+x^{1/2})$. By (\ref{dg}), we can differentiate on $t$ in both sides of (\ref{Dtz_k=Dtz_k}) for $n$ times and get
$$(\log\Delta_\CA)^n\Delta_\CA^t z_k\xi_\CA=P^*(\log\Delta_\CB)^n\Delta_\CB^t z_k\xi_\CB.$$
Since $\Delta^{is}=\sum_{n=0}^{\wx}(is\log \Delta)^n/n!$ for any $s\in\R$, we have
$$\Delta_\CA^{is}\Delta_\CA^t z_k\xi_\CA=P^*\Delta_\CB^{is}\Delta_\CB^t z_k\xi_\CB.$$
Let $t\to 0$, we have
$$\Delta_\CA^{is}  z_k\xi_\CA=P^*\Delta_\CB^{is}  z_k\xi_\CB.$$
Hence
$$\lag (1-PP^*)\Delta_\CB^{is}  z_k\xi_\CB,\Delta_\CB^{is}  z_k\xi_\CB\rag=\Vert \Delta_\CB^{is} z_k\xi_\CB\Vert^2-\Vert\Delta_\CA^{is}  z_k\xi_\CA\Vert^2=\Vert z_k\Vert_{2,\tau}^2-\Vert z_k\Vert_{2,\tau}^2=0.$$
Since $(1-PP^*)\geq 0$, we must have $(1-PP^*)\Delta_\CB^{is}  z_k\xi_\CB=0$.
And
\begin{equation}\label{PDAit=DBit}
P\Delta_\CA^{is}  z_k\xi_\CA=PP^*\Delta_\CB^{is}  z_k\xi_\CB=\Delta_\CB^{is}  z_k\xi_\CB.
\end{equation}
Since $\mathrm{span}\{z_k\hat{1}\}$ is dense in $L^2(M)$, we have $P\Delta_\CA^{is}  z\xi_\CA=\Delta_\CB^{is}  z\xi_\CB$ for any $z\in M$. That is, $\CP(\sigma_t^{\CA}(z))\xi_\CB=\sigma_t^\CB(z)\xi_\CB$.
Hence $$\CP(\sigma_s^{\CA}(z))=\sigma_s^\CB(z)$$
for any $z\in M$ and $s\in \R$. 

Note that
$$\CP(\sigma_s^{\CA}(z)^*\sigma_s^{\CA}(z))=\CP(\sigma_s^{\CA}(z^*z))=\sigma_s^\CB(z^*z)=\sigma_s^\CB(z)^*\sigma_s^\CB(z)=\CP(\sigma_s^{\CA}(z)^*)\CP(\sigma_s^{\CA}(z)).$$
Hence $\sigma_s^{\CA}(z)$ is contained in the multiplicative domain $\mathrm{mult}(\CP)$ of $\CP$. By \cite[Proposition 2.5.11]{AP17}, since $\CP$ preserves the normal faithful states, $\CP$ is normal faithful and $\mathrm{mult}(\CP)$ is weakly closed. Therefore, $$\CA_\mathrm{RN}=\lag \sigma_s^{\CA}(z)\rag_{z\in M,s\in\R}\subset \mathrm{mult}(\CP).$$
And $\CP|_{\CA_\mathrm{RN}}:(\CA_\mathrm{RN},\varphi_{\CA_\RN})\to (\CP(\CA_\mathrm{RN}),\varphi_\CB|_{\CP(\CA_\mathrm{RN})})$ is a state preserving $\ast$-isomorphism. Therefore, 
$$h_\varphi(\CB,\varphi_\CB)=h_\varphi(\CA,\varphi_\CA)=h_\varphi(\CA_\mathrm{RN},\varphi_{\CA_\RN})=h_\varphi(\CP(\CA_\mathrm{RN}),\varphi_\CB|_{\CP(\CA_\mathrm{RN})}).$$
By \cite[Theorem 3.6]{Zh23}, we must have $\CB_\mathrm{RN}=(\CP(\CA_\mathrm{RN}))_\mathrm{RN}=\CP(\CA_\mathrm{RN})$. And $\CP|_{\CA_\mathrm{RN}}:\CA_\mathrm{RN}\to\CB_\mathrm{RN}$ is a $\ast$-isomorphism.
\end{proof}

Let $\Gamma$ be a countable discrete group. Then by \cite[Example 3.5]{Zh23}, for any nonsingular $\Gamma$-space $(X,\nu_X)$, we have that the Radon-Nikodym factor of the crossed product $L(\Gamma\car (X,\nu_X))$ with respect to $L(\Gamma)$ is exactly $L(\Gamma\car(X_\mathrm{RN},\nu_{X_\mathrm{RN}}))$. Then as a direct corollary of this fact and Theorem \ref{Thm NC}, we have the following Corollary \ref{Thm Gamma}.
\begin{corollary}\label{Thm Gamma}
Let $\mu\in\prob(\Gamma)$ be a generating measure. Let $(X,\nu_X)$ and $(Y,\nu_Y)$ be $(\Gamma,\mu)$-spaces. Assume that there exists a measure preserving $\Gamma$-equivariant ucp map $\CP:L^\wx(X,\nu_X)\to L^\wx(Y,\nu_Y)$, i.e., $\nu_X=\nu_Y\circ\CP$ as states on $L^\wx(X)$. Then we have
$$h_\mu(X,\nu_X)\leq h_\mu(Y,\nu_Y).$$
Assume that $h_\mu(Y,\nu_Y)<+\wx$. Then the equality holds if and only if $\CP|_{L^\wx(X_{\mathrm{RN}})}:L^\wx(X_{\mathrm{RN}})\to L^\wx(Y_{\mathrm{RN}})$ is a $\ast$-isomorphism.
\end{corollary}

Following a proof in the same spirit as that of Theorem \ref{Thm NC}, we can generalize Corollary \ref{Thm Gamma} to the setting of locally compact groups.

Let $G$ be a locally compact second countable group and $\mu\in\prob(G)$ be an \textbf{admissible} measure, i.e., there exists $k\geq1$ such that $\mu^{*k}$ is absolutely continuous with respect to the left Haar measure $m_G$ and $\cup_{n\geq 1}(\supp \mu)^n=G$.

\begin{theorem}\label{P for G}
Let $(X,\nu_X)$ and $(Y,\nu_Y)$ be $(G,\mu)$-spaces. Assume that there exists a measure preserving $G$-equivariant ucp map $\CP:L^\wx(X,\nu_X)\to L^\wx(Y,\nu_Y)$. Then we have
$$h_\mu(X,\nu_X)\leq h_\mu(Y,\nu_Y).$$
Assume that $h_\mu(Y,\nu_Y)<+\wx$. Then the equality holds if and only if $\CP|_{L^\wx(X_{\mathrm{RN}})}:L^\wx(X_{\mathrm{RN}})\to L^\wx(Y_{\mathrm{RN}})$ is a $\ast$-isomorphism.
\end{theorem}
\begin{proof}
Since $\mu$ is admissible, there exists $k_0\geq 1$ with $\mu^{*k_0}\prec m_G$. Hence for any $l\geq 0$, $\mu^{*(k_0+l)}\prec \mu^{*l}\ast m_G=m_G$. Let $\bar{\mu}=\sum_{l\geq0}2^{-l-1}\mu^{*(k_0+l)}$. Then $\bar{\mu}\prec m_G$. Since $\cup_{n\geq 1}(\supp \mu)^n=G$, there exists $n_0\geq 1$ with $e\in (\supp \mu)^{n_0}$. Hence $e\in (\supp \mu)^{n_0k_0}$, and for any $n\geq 1$,
$$(\supp \mu)^{n}\subset(\supp \mu)^{n_0k_0+n}\subset \supp \mu^{*(n_0k_0+n)}\subset\supp\bar{\mu}.$$
Therefore, $G=\cup_{n\geq 1}(\supp \mu)^n\subset\supp\bar{\mu}$. Also, for any $(G,\mu)$-space $(Z,\nu_Z)$, by \cite[Proposition 2.24]{Fur02} and \cite[Corollary 1]{KV83}, we have 
$$h_{\bar{\mu}}(Z,\nu_Z)=\sum_{l\geq0}2^{-l-1}(k_0+l)\cdot h_\mu(Z,\nu_Z):= C\cdot h_\mu(Z,\nu_Z).$$ 
Hence after replacing $\mu$ with $\bar{\mu}$, we may assume that $\mu\prec m_G$ and $\supp\mu=G$. 

Let $L^2(G)=L^2(G,m_G)$. Since $\CP:L^\wx(X,\nu_X)\to L^\wx(Y,\nu_Y)$ is measure preserving, it can be extended to $\bar{\CP}:L^2(X,\nu_X)\to L^2(Y,\nu_Y)$ with $\Vert\bar{\CP}\Vert=1$. Let
$$P=1\otimes\bar{\CP}: L^2(G)\otimes L^2(X)=L^2(G\times X)\to L^2(G)\otimes L^2(Y)=L^2(G\times Y).$$
That is, for $F\in L^2(G\times X)$ (note that $F(g,\,\cdot\,)\in L^2(X)$ for $m_G$-a.e. $g\in G$),
$$P(F)(g,y)=\bar{\CP}(F(g,\,\cdot\,))(y)\ ((g, y)\in G\times Y).$$
Then $\Vert P\Vert=1$. 

Let 
$$\Delta_X(g,x)=\frac{\d g^{-1}\nu_X}{\d \nu_X}(x) \mbox{ and } \Delta_Y(g,y)=\frac{\d g^{-1}\nu_Y}{\d \nu_Y}(y)$$ 
be the Radon-Nikodym cocycles of $(X,\nu_X)$ and $(Y,\nu_Y)$. Then we can view $\Delta_X$ and $\Delta_Y$ as unbounded positive operators on $L^2(G\times X)$ and $L^2(G\times Y)$.

Let $\Delta_G:G\to \R_+^*$ be the modular function of $G$ and define the unbounded anti-linear operator $S_X:L^2(G\times X)\to L^2(G\times X)$ by
\begin{equation}\label{def SX}
S_X(F)(g,x)=\Delta_G(g)^{1/2} \cdot\overline{F(g^{-1},gx)}.
\end{equation}
Then we have $S_X^2=1$ and
$$\lag \Delta_X F,F\rag=\int_G\int_X|F(g,g^{-1}x)|^2\d\nu_X(x)\d m_G(g)=\lag S_XF,S_XF\rag.$$
Hence $\Delta_X=S_X^*S_X$. In the case that $G$ is discrete, if $(M,\tau)\subset(\CA,\varphi_\CA)$ is taken as $L(G)\subset L(\Gamma\car(X,\nu_X))$, we will have $S_X=S_\CA:a\xi_\CA\mapsto a^*\xi_\CA$ $(a\in \CA)$ and $\Delta_X=\Delta_\CA$ exactly, which is the reason that we define $S_X$ as in (\ref{def SX}).

Similarly, we can define $S_Y$ and have $\Delta_Y=S_Y^*S_Y$. Note that we also have $PS_X=S_Y P$. Hence
$$\Delta_X=S_X^*S_X\geq S_X^*P^*PS_X=P^*S_Y^*S_Y P=P^*\Delta_Y P.$$
By Lemma \ref{P Dt P Dt}, for $0\leq t\leq 1$, we have
\begin{equation}\label{DY t leq DX t}
 P^*\Delta_Y^t P\leq \Delta_X^t.
\end{equation}
Let $e^X_G: L^2(G\times X)\to L^2(G)$ and $e^Y_G: L^2(G\times Y)\to L^2(G)$ be the orthogonal projections. That is, for $F_1\in L^2(G\times X)$ and $F_2\in L^2(G\times Y)$,
$$e^X_G(F_1)(g)=\int_X F_1(g,x)\d\nu_X(x)\mbox{ and }e^Y_G(F_2)(g)=\int_Y F_2(g,y)\d\nu_Y(y).$$
Then we have $Pe^X_G=e^Y_G$. Hence by (\ref{DY t leq DX t}), we have
$$ e^Y_G\Delta_Y^t e^Y_G\leq e^X_G\Delta_X^t e^X_G.$$
Therefore,
$$ e^Y_G\frac{\Delta_Y^t-1}{t} e^Y_G\leq e^X_G\frac{\Delta_X^t-1}{t} e^X_G.$$
Let $t\to 0$, we have
$$e^Y_G\log\Delta_Y e^Y_G\leq e^X_G \log\Delta_X e^X_G.$$
That is,
$$\int_Y\log\Delta_Y(g,y)\d\nu_Y(y)\leq \int_X\log\Delta_X(g,x)\d\nu_X(x)\mbox{ for } m_G\mbox{-a.e. } g\in G.$$
Therefore,
\begin{align*}
h_\mu(Y,\nu_Y)=&-\int_G\left(\int_Y\log\Delta_Y(g,y)\d\nu_Y(y))\right)\d \mu(g)\\
\geq&-\int_G\left(\int_X\log\Delta_X(g,x)\d\nu_X(x)\right)\d \mu(g)=h_\mu(X,\nu_X).
\end{align*}

From now on, we assume that $h_\mu(Y,\nu_Y)<+\wx$.

Assume that $\CP|_{L^\wx(X_{\mathrm{RN}})}:L^\wx(X_{\mathrm{RN}})\to L^\wx(Y_{\mathrm{RN}})$ is a $\ast$-isomorphism. Then $\CP|_{L^\wx(X_{\mathrm{RN}})}$ induces a measure preserving $G$-equivariant isomorphism $\beta_\mathrm{RN}:(Y_\mathrm{RN},\nu_{Y_\mathrm{RN}})\to(X_\mathrm{RN},\nu_{X_\mathrm{RN}})$. By \cite[Lemma 1.13]{NZ00}, the Radon-Nikodym factor map preserves the Furstenberg entropy. Hence we have 
$$h_\mu(X,\nu_X)=h_\mu(X_\mathrm{RN},\nu_{X_\mathrm{RN}})=h_\mu(Y_\mathrm{RN},\nu_{Y_\mathrm{RN}})=h_\mu(Y,\nu_Y).$$

Conversely, assume that $h_\mu(X,\nu_X)=h_\mu(Y,\nu_Y)<+\wx$. Since $\mu\prec m_G$, we can take $f=(\frac{\d \mu}{\d m_G})^{1/2}\in L^2(G)$. Since $|\log x|\leq 2x-\log x$ for $x>0$, we have
\begin{align*}
&\lag|\log \Delta_X| f\otimes \mathds{1}_X, f\otimes \mathds{1}_X\rag\\
\leq &\lag(2\Delta_X-\log \Delta_X) f\otimes \mathds{1}_X, f\otimes \mathds{1}_X\rag\\
=&2\Vert f\Vert_2^2+h_\mu(X,\nu_X)\\
<&+\wx.
\end{align*}
Hence $f\otimes \mathds{1}_X\in D(|\log \Delta_X|^{1/2})\cap D(\Delta_X^{1/2})$. For the same reason, we have $f\otimes \mathds{1}_Y\in D(|\log \Delta_Y|^{1/2})\cap D(\Delta_Y^{1/2})$.

We also have
$$-\lag\log\Delta_Y f\otimes \mathds{1}_Y, f\otimes \mathds{1}_Y\rag  =h_\mu(Y,\nu_Y)=h_\mu(X,\nu_X)= -\lag\log\Delta_X f\otimes \mathds{1}_X, f\otimes \mathds{1}_X\rag.$$
Therefore, as an analogue of (\ref{log DB=log DA}), we have
\begin{equation}\label{log Y = log X}
\lag\log\Delta_Y f\otimes \mathds{1}_Y, f\otimes \mathds{1}_Y\rag  =\lag\log\Delta_X f\otimes \mathds{1}_X, f\otimes \mathds{1}_X\rag.
\end{equation}
Following the exact same discussion before (\ref{PDAit=DBit}) in the proof of Theorem \ref{Thm NC}, with $(\Delta_\CA,\Delta_\CB,z_k\xi_\CA,z_k\xi_\CB)$ replaced by $(\Delta_X,\Delta_Y,f\otimes \mathds{1}_X,f\otimes \mathds{1}_Y)$, we can prove the following equality as an analogue of (\ref{PDAit=DBit}):
$$P\Delta_X^{it}(f\otimes \mathds{1}_X)=\Delta_Y^{it}(f\otimes \mathds{1}_Y) \mbox{ for any $t\in \R$}.$$
That is, for any $t\in\R$
$$f(g)\cdot \CP(\Delta_X(g,\,\cdot\,)^{it})(y)=f(g)\cdot\Delta_Y(g,y)^{it}\mbox{ for $m_G\times \nu_Y$-a.e. $(g,y)\in G\times Y$.}$$
Hence
$$\begin{aligned}
&\int_G\left(\int_Y|\CP(\Delta_X(g,\,\cdot\,)^{it})(y)-\Delta_Y(g,y)^{it}|^2\d \nu_Y(y)\right)\d \mu(g)\\
=&\int_G\left(\int_Y|f(g)\cdot\CP(\Delta_X(g,\,\cdot\,)^{it})(y)-f(g)\cdot\Delta_Y(g,y)^{it}|^2\d \nu_Y(y)\right)\d m_G(g)\\
=&0.
\end{aligned}$$
Therefore, there exists a $\mu$-conull subset $U_t\subset G$ such that
\begin{equation}\label{P DXit = DYit}
\CP(\Delta_X(g,\,\cdot\,)^{it})=\Delta_Y(g,\,\cdot\,)^{it} \mbox{ in $L^\wx(Y,\nu_Y)$ for any $g\in U_t$.}
\end{equation}
Since $\supp\mu=G$, the $\mu$-conull subset $U_t\subset G$ is also a dense subset of $G$. Then by the proof of \cite[Proposition 1.14]{NZ00}, for any $g_0\in G\setminus U_t$, there exists a sequence $(g_n)\subset U_t$ such that $g_n\to g_0$ and $\Delta_X(g_n,x)^{it}\to\Delta_X(g_0,x)^{it}$ $\nu_X$-a.e. There also exists a subsequence $(g_{n_k})$ of $(g_n)$ such that $\Delta_Y(g_{n_k},y)^{it}\to\Delta_X(g_{0},y)^{it}$ $\nu_Y$-a.e. Moreover, since $\{\Delta_X(g_n,\,\cdot\,)^{it}\}$ are uniformly bounded, we have $\Delta_X(g_{n_k},\,\cdot\,)^{it}\xrightarrow{\mathrm{wo}}\Delta_X(g_0,\,\cdot\,)^{it}$ within $L^\wx(X)$. For the same reason, $\Delta_Y(g_{n_k},\,\cdot\,)^{it}\xrightarrow{\mathrm{wo}}\Delta_Y(g_0,\,\cdot\,)^{it}$. Since $\CP$ is normal for preserving the normal faithful states \cite[Proposition 2.5.11]{AP17}, we know that (\ref{P DXit = DYit}) holds for any $g\in G$.

Since 
$$\begin{aligned}
&\CP(\Delta_X(g,\,\cdot\,)^{-it}\cdot\Delta_X(g,\,\cdot\,)^{it})\\
=&1\\
=&\Delta_Y(g,\,\cdot\,)^{-it}\cdot\Delta_Y(g,\,\cdot\,)^{it}\\
=&\CP(\Delta_X(g,\,\cdot\,)^{-it})\cdot\CP(\Delta_X(g,\,\cdot\,)^{it}),
\end{aligned}$$
we have $L^\wx(X_{\mathrm{RN}})=\lag \Delta_X(g,\,\cdot\,)^{it}\rag_{g\in G,  t\in \R}\subset \mathrm{mult} (\CP)$. Also note that $L^\wx(Y_{\mathrm{RN}})=\lag \Delta_Y(g,\,\cdot\,)^{it}\rag_{g\in G,  t\in \R}$. Therefore, $\CP|_{L^\wx(X_{\mathrm{RN}})}:L^\wx(X_{\mathrm{RN}})\to L^\wx(Y_{\mathrm{RN}})$ is a $\ast$-isomorphism.
\end{proof}

We restate Theorem \ref{P for G} in terminology that aligns more closely with ergodic group theory.
\begin{theorem}\label{beta for G}
Let $(X,\nu_X)$ be a compact metrizable $(G,\mu)$-space and $(Y,\nu_Y)$ be a $(G,\mu)$-space. Assume that there exists a quasi-factor map, i.e., a $G$-equivariant measurable map $\beta:Y\to \prob(X)$ with $\nu_X=\int_Y\beta_y\d\nu_Y(y)$. Then we have
$$h_\mu(X,\nu_X)\leq h_\mu(Y,\nu_Y).$$
Assume that $h_\mu(Y,\nu_Y)<+\wx$. Then the equality holds if and only if $\beta$ induces a measure preserving $G$-equivariant isomorphism $\beta_\mathrm{RN}:(Y_\mathrm{RN},\nu_{Y_\mathrm{RN}})\to(X_\mathrm{RN},\nu_{X_\mathrm{RN}})$ with $\pi_{X_{\mathrm{RN}}*}\circ\beta(y)=\delta_{\beta_\mathrm{RN}\circ\pi_{Y_\mathrm{RN}}(y)}$ for $\nu_Y$-a.e. $y\in Y$, i.e., the following commutative diagram:
$$\begin{tikzcd}
Y \arrow[d, "\pi_{Y_\mathrm{RN}}", two heads] \arrow[r, "\beta"] & \prob(X) \arrow[d, "\pi_{X_\mathrm{RN}*}", two heads] \\
Y_\mathrm{RN} \arrow[r, "\beta_\mathrm{RN}"] \arrow[r, "\sim"']  & X_\mathrm{RN}\subset \prob(X_\mathrm{RN}).            
\end{tikzcd}$$
\end{theorem}
\begin{proof}
Define a $G$-equivariant ucp map $\CP:C(X)\to L^\wx(Y,\nu_Y)$ by
$$\CP(f)(y)=\beta_y(f)\ (y\in Y, \ f\in C(X)).$$
Since $\nu_X=\int_Y\beta_y\d\nu_Y(y)$, we have $\nu_X=\nu_Y\circ\CP$ as states on $C(X)$. Hence $\CP$ can be naturally extended to a measure preserving ucp map $\CP:L^\wx(X,\nu_X)\to L^\wx(Y,\nu_Y)$. Then by Theorem \ref{P for G}, we have 
$$h_\mu(X,\nu_X)\leq h_\mu(Y,\nu_Y).$$

From now on, assume that $h_\mu(Y,\nu_Y)<+\wx$.

Assume that there exists a measure preserving $G$-equivariant isomorphism $\beta_\mathrm{RN}:(Y_\mathrm{RN},\nu_{Y_\mathrm{RN}})\to(X_\mathrm{RN},\nu_{X_\mathrm{RN}})$. Then by \cite[Lemma 1.13]{NZ00}, we have 
$$h_\mu(X,\nu_X)=h_\mu(X_\mathrm{RN},\nu_{X_\mathrm{RN}})=h_\mu(Y_\mathrm{RN},\nu_{Y_\mathrm{RN}})=h_\mu(Y,\nu_Y).$$

Conversely, assume that $h_\mu(X,\nu_X)=h_\mu(Y,\nu_Y)$. Then by Theorem \ref{P for G}, $\CP|_{L^\wx(X_{\mathrm{RN}})}:L^\wx(X_\mathrm{RN},\nu_{X_\mathrm{RN}})\to L^\wx(Y_\mathrm{RN},\nu_{Y_\mathrm{RN}})$ is a measure preserving $\ast$-isomorphism, which induces a measure preserving $G$-equivariant isomorphism $\beta_\mathrm{RN}:(Y_\mathrm{RN},\nu_{Y_\mathrm{RN}})\to(X_\mathrm{RN},\nu_{X_\mathrm{RN}})$. And the commutative diagram
$$\begin{tikzcd}
L^\wx(X) \arrow[r, "\CP"]                                                                         & L^\wx(Y)                                                      \\
L^\wx(X_\mathrm{RN}) \arrow[r, "\CP"] \arrow[u, "\pi_{X_\mathrm{RN}}^*", hook] \arrow[r, "\sim"'] & L^\wx(Y_\mathrm{RN}) \arrow[u, "\pi_{Y_\mathrm{RN}}^*", hook]
\end{tikzcd}$$
is exactly equivalent to
$$\begin{tikzcd}
Y \arrow[d, "\pi_{Y_\mathrm{RN}}", two heads] \arrow[r, "\beta"] & \prob(X) \arrow[d, "\pi_{X_\mathrm{RN}*}", two heads] \\
Y_\mathrm{RN} \arrow[r, "\beta_\mathrm{RN}"] \arrow[r, "\sim"']  & X_\mathrm{RN}\subset \prob(X_\mathrm{RN}).            
\end{tikzcd}$$
\end{proof}

\section{Rigidity of unique stationary Poisson boundaries}
In this section, we apply Theorem \ref{Thm A} and Theorem \ref{Thm B} to obtain more rigidity results. We still fix a normal regular strongly generating hyperstate $\varphi\in\CS_\tau(\BL)$ with Poisson boundary $(\CB_\varphi,\zeta)$, and an admissible measure $\mu\in\prob(G)$ with Poisson boundary $(B,\nu_B)$. 

\begin{definition}\label{def RN irr}
We say that a W$^*$-extension $(\CA,\varphi_\CA)$ of $(M,\tau)$ is \textbf{RN-irreducible} if $(\CA_\RN,\varphi_{\CA_\RN})=(\CA,\varphi_\CA)$. A nonsingular $G$-space $(X,\nu_X)$ is \textbf{RN-irreducible} if $(X_\RN,\nu_{X_\RN})=(X,\nu_X)$. 
\end{definition}

\begin{example}
By the definition of Radon-Nikodym factors (Subsection \ref{RN factor def}), any (classical or noncommutative) Radon-Nikodym factor itself is RN-irreducible. By the uniqueness of (classical and noncommutative \cite[Theorem 3.2(1)]{Zh23b}) Furstenberg boundary maps, any \textbf{$\varphi$-boundary}, i.e., intermediate subalgebra of $(M,\tau)\subset(\CB_\varphi,\zeta)$, or \textbf{$\mu$-boundary}, i.e., $G$-factor of $(B,\nu_B)$, is also RN-irreducible. 
\end{example}

As a corollary of Theorem \ref{P for G}, we provide a new proof to \cite[Theorem 4.4]{FG10}, which characterizes stationary spaces with maximal entropy, and the original proof used the tool of standard covers.
\begin{corollary}\label{max entropy for G}
Assume that $h_\mu(B,\nu_B)<+\wx$. Let $(X,\nu_X)$ be a $(G,\mu)$-space. Then $h_\mu(X,\nu_X)=h_\mu(B,\nu_B)$ if and only if $(X_\RN,\nu_{X_\RN})\cong(B,\nu_B)$ as $(G,\mu)$-spaces.
\end{corollary}
\begin{proof}
The ``if'' direction is clear. 

Assume that $h_\mu(X,\nu_X)=h_\mu(B,\nu_B)$. After endowing a compact metrizable model to $(X,\nu_X)$, we may assume that $X$ is a compact metrizable $G$-space. By \cite{Fu63,Fu63a}, there exists the Furstenberg boundary map $\beta:B\to\prob(X)$ with $\nu_X=\int_B\beta_b\nu_B(b)$. Then by Theorem \ref{Thm B}, we have $h_\mu(X,\nu_X)\leq h_\mu(B,\nu_B)$. Since the equality holds, we know that there exists a $G$-isomorphism $(X_\RN,\nu_{X_\RN})\cong(B_\RN,\nu_{B_\RN})=(B,\nu_B)$.
\end{proof}

As a corollary of Theorem \ref{Thm A} and the noncommutative analogue of Corollary \ref{max entropy for G} \cite[Theorem 4.4]{FG10}, we have the following Corollary \ref{max entropy for M}.

\begin{corollary}\label{max entropy for M}
Assume that $h_\varphi(\CB_\varphi,\zeta)<+\wx$. Let $(\CA,\varphi_\CA)$ be a $\varphi$-stationary W$^*$-extension of $(M,\tau)$. Then $h_\varphi(\CA,\varphi_A)=h_\varphi(\CB_\varphi,\zeta)$ if and only if $(\CB_\varphi,\zeta)\subset(\CA,\varphi_\CA)$ as the Radon-Nikodym factor.
\end{corollary}
\begin{proof}
The ``if'' direction is clear. 

Assume that $h_\varphi(\CA,\varphi_A)=h_\varphi(\CB_\varphi,\zeta)$. Since $\varphi_\CA$ is $\varphi$-stationary, by \cite[Theorem 3.2]{Zh23b}, there exists a noncommutative Furstenberg boundary map $\Phi_{\varphi_\CA}\in\UCP_M(\CA,\CB_\varphi)$ with $\varphi_\CA=\zeta\circ\Phi_{\varphi_\CA}$, where $\UCP_M(\CA,\CB_\varphi)$ is the set of $M$-bimodular ucp maps between $\CA$ and $\CB_\varphi$. Hence by Theorem \ref{Thm A}, we have $h_\varphi(\CA,\varphi_A)\leq h_\varphi(\CB_\varphi,\zeta)$. Moreover, since the equality holds, we have $(\CA_\mathrm{RN},\varphi_{\CA_\RN})\cong (\CB_\varphi,\zeta)_\RN=(\CB_\varphi,\zeta)$ as W$^*$-extensions of $(M,\tau)$.
\end{proof}

\begin{definition}\label{unique stationary}
Following \cite[Definition 4.4]{Zh24}, for a $\varphi$-stationary W$^*$-inclusion $(M,\tau)\subset(\CA,\varphi_\CA)$, we say that $(\CA,\varphi_\CA)$ is \textbf{$\varphi$-unique stationary} if there exists a weakly dense C$^*$-subalgebra $\CA_0\subset\CA$ with $M\subset\CA_0$ such that $\varphi_\CA|_{\CA_0}$ is the unique $\varphi$-stationary hyperstate on $\CA_0$. 

A $(G,\mu)$-space $(X,\nu_X)$ is \textbf{$\mu$-unique stationary} if it admits a compact metrizable model $(\bar{X},\nu_{\bar{X}})$ such that $\nu_{\bar{X}}\in \prob(\bar{X})$ is the unique $\mu$-stationary Borel probability measure on $\bar{X}$.
\end{definition}

The following example is \cite[Example 4.6]{Zh24}
\begin{example}\label{example USB}
For a countable discrete group $\Gamma$ and a generating measure $\mu\in\prob(\Gamma)$, let $(M,\tau)=L(\Gamma)$ and $\varphi (T)=\sum_{\gamma\in\Gamma}\mu(\gamma)\langle T\mathds{1}_{\gamma},\mathds{1}_{\gamma}\rangle$. Assume that a $\mu$-boundary $(B_0,\nu_0)$ is $\mu$-unique stationary with unique stationary compact metrizable model $\bar{B}_0$. Then $L(\Gamma\car (B_0,\nu_0))$ is also $\varphi$-unique stationary with unique stationary weakly dense C$^*$-subalgebra $\CA_0=\mathrm{C}^*(L(\Gamma),C(\bar{B}_0))$. Here $C(\bar{B}_0)\subset L^\wx(B_0,\nu_0)$ as a C$^*$-subalgebra.
\end{example}

\begin{definition}\label{weak amenability}
Following \cite[Definition 5.11]{Zh24}, a W$^*$-extension $(\CB,\varphi_\CB)$ of $(M,\tau)$ is \textbf{weakly amenable} if for any C$^*$-inclusion $M\subset \mathcal{C}$, there exists a $M$-bimodular ucp map $\CP:\mathcal{C}\to\CB$.

Following \cite[Definition 2.14]{HK24}, a nonsingular $G$-space $(Y,\nu_Y)$ is \textbf{weakly amenable} if for any $G$-C$^*$-algebra $A$ (or equivalently, any $A=C(X)$ for a compact metrizable $G$-space $X$), there exists a $G$-equivariant ucp map $P:A\to L^\wx(Y,\nu_Y)$.

Following \cite[Definition 2.14]{HK24} and \cite[ Proposition 5.12]{Zh24}, amenability implies weak amenability in both settings.
\end{definition}

As a corollary of Theorem \ref{Thm B}, we can show that there is an entropy separation between unique stationary spaces and (weakly) amenable spaces.
\begin{corollary}\label{us leq am for G}
Let $(X,\nu_X)$ be a $\mu$-unique stationary $(G,\mu)$-space and $(Y,\nu_Y)$ be a weakly amenable $(G,\mu)$-space. Then we have
$$h_\mu(X,\nu_X)\leq h_\mu(Y,\nu_Y).$$
In particular, let
$$h_\mathrm{us}(\mu)=\sup\{h_\mu(X,\nu_X)\mid\mbox{$(X,\nu_X)$ is a $\mu$-unique stationary $(G,\mu)$-space}\},$$
$$h_\mathrm{wa}(\mu)=\inf\{h_\mu(Y,\nu_Y)\mid\mbox{$(Y,\nu_Y)$ is a weakly amenable $(G,\mu)$-space}\},$$
$$h_\mathrm{am}(\mu)=\inf\{h_\mu(Z,\nu_Z)\mid\mbox{$(Z,\nu_Z)$ is an amenable $(G,\mu)$-space}\}.$$
Then we have
$$h_\mathrm{us}(\mu)\leq h_\mathrm{wa}(\mu)\leq h_\mathrm{am}(\mu).$$
Moreover, assume that $h_\mu(B,\nu_B)<+\wx$, then up to isomorphisms, $G$ admits at most one RN-irreducible $(G,\mu)$-space that is $\mu$-unique stationary and weakly amenable.
\end{corollary}
\begin{proof}
After endowing a compact metrizable model to $(X,\nu_X)$, we may just assume that $X$ is a compact metrizable $G$-space with the unique $\mu$-stationary probability measure $\nu_X$. Since $G\car(Y,\nu_Y)$ is weakly amenable, there exists a $G$-equivariant ucp map 
$\CP:C(X)\to L^\wx(Y,\nu_Y)$. Since $(X,\nu_X)$ is $\mu$-unique stationary, we must have $\nu_X=\nu_Y\circ \CP$ as states on $C(X)$. Hence $\CP$ can be extended to a measure preserving ucp map $\CP:L^\wx(X,\nu_X)\to L^\wx(Y,\nu_Y)$. Therefore, by Theorem \ref{P for G}, we have $h_\mu(X,\nu_X)\leq h_\mu(Y,\nu_Y)$.

By the arbitrariness of $(X,\nu_X)$ and $(Y,\nu_Y)$, we have $h_\mathrm{us}(\mu)\leq h_\mathrm{wa}(\mu)$. Since amenability implies weak amenability, we also have $h_\mathrm{wa}(\mu)\leq h_\mathrm{am}(\mu)$,

Assume that $h_\mu(B,\nu_B)<+\wx$. Let $(Z,\nu_{Z})$ and $(W,\nu_W)$ be $\mu$-unique stationary weakly amenable RN-irreducible $(G,\mu)$-spaces. Since $(Z,\nu_{Z})$ is unique stationary and $(W,\nu_W)$ is weakly amenable, by the same discussion above, there exists a measure preserving $G$-equivariant ucp map $\CP_1: L^\wx(Z,\nu_{Z})\to L^\wx(W,\nu_W)$ and 
$$h_\mu(Z,\nu_Z)\leq h_\mu(W,\nu_W)\leq h_\mu(B,\nu_B)<+\wx.$$ 
For the same reason, we also have $h_\mu(W,\nu_W)\leq h_\mu(Z,\nu_Z)$. Hence we must have $h_\mu(Z,\nu_Z)= h_\mu(W,\nu_W)$ and by Theorem \ref{Thm B},
$$(Z,\nu_Z)=(Z_\RN,\nu_{Z_\RN})\cong(W_\RN,\nu_{W_\RN})=(W,\nu_W).$$
\end{proof}

We have the following Corollary \ref{us leq am for M} as the noncommutative analogue of Corollary \ref{us leq am for G}.
\begin{corollary}\label{us leq am for M}
Let $(\CA,\varphi_\CA)$ and $(\CB,\varphi_\CB)$ be $\varphi$-stationary W$^*$-extensions of $(M,\tau)$. Assume that $(\CA,\varphi_\CA)$ is $\varphi$-unique stationary and $\CB$ is amenable. Then 
$$h_\varphi(\CA,\varphi_\CA)\leq h_\varphi(\CB,\varphi_\CB).$$
In particular, let
$$h_\mathrm{us}(\varphi)=\sup\{h_\varphi(\CA,\varphi_\CA)\mid\mbox{$(\CA,\varphi_\CA)$ is a $\varphi$-unique stationary W$^*$-extension}\},$$
$$h_\mathrm{wa}(\varphi)=\inf\{h_\varphi(\CB,\varphi_\CB)\mid\mbox{$(\CB,\varphi_\CB)$ is a weakly amenable $\varphi$-stationary W$^*$-extension}\},$$
$$h_\mathrm{am}(\varphi)=\inf\{h_\varphi(\mathcal{C},\varphi_\mathcal{C})\mid\mbox{$(\mathcal{C},\varphi_\mathcal{C})$ is an amenable $\varphi$-stationary W$^*$-extension}\}.$$
Then we have
$$h_\mathrm{us}(\varphi)\leq h_\mathrm{am}(\varphi)\leq h_\mathrm{am}(\varphi).$$
Moreover, assume that $h_\varphi(\CB_\varphi,\zeta)<+\wx$, then up to $\ast$-isomorphisms, $(M,\tau)$ admits at most one RN-irreducible W$^*$-extension that is $\varphi$-unique stationary and weakly amenable.
\end{corollary}
\begin{proof}
Let $\CA_0\subset\CA$ be the weakly dense C$^*$-subalgebra with a unique $\varphi$-stationary hyperstate. Since $(\CB,\varphi_\CB)$ is weakly amenable with respect to $M$, there exists a ucp map $\CP\in\UCP_M(\CA_0,\CB)$. Since $\varphi_\CA|_{\CA_0}$ is the unique $\varphi$-stationary hyperstate on $\CA_0$, we must have $\varphi_\CA|_{\CA_0}=\varphi_\CB\circ\CP$. Hence $\CP$ can be naturally extended to $(\CA,\varphi_\CA)\cong\overline{(\CA_0,\varphi_\CA|_{\CA_0})}^\mathrm{wo}$, i.e., the weak closure of $\CA_0$ under the GNS construction with respect to $\varphi_\CA|_{\CA_0}$. Hence we have a state preserving $M$-bimodular ucp map $\CP:(\CA,\varphi_\CA)\to (\CB,\varphi_\CB)$. And the rest of the proof follows from a similar argument as in the proof of Corollary \ref{us leq am for G}, with Theorem \ref{Thm A} applied.
\end{proof}

The relationship between amenability and the unique stationary property was also studied in \cite{HK24}. As a corollary of Corollary \ref{us leq am for G} and \ref{us leq am for M}, we provide new proofs to \cite[Theorem 4.8 and Corollary 4.17]{HK24} from an entropy perspective. For the definition of \textbf{standard covers}, we refer to \cite{FG10}.

\begin{corollary}\label{HK24 Thm A}
Assume that $h_\mu(B,\nu_B)<+\wx$. Let $(Y,\nu_Y)$ be a $(G,\mu)$-space with standard cover $(\tilde{Y},\nu_{\tilde{Y}})$ realized by the factor map $\pi_Y:(\tilde{Y},\nu_{\tilde{Y}})\to(Y,\nu_Y)$ and the relatively measure-preserving factor map $\pi_{B_0}: (\tilde{Y},\nu_{\tilde{Y}})\to(B_0,\nu_0)$, where $(B_0,\nu_0)$ is a $\mu$-boundary. Let $(Z,\nu_Z)$ be a weakly amenable $(G,\mu)$-space, and $(\tilde{Y},\nu_{\tilde{Y}})\xrightarrow{\pi_Z}(Z,\nu_Z)\xrightarrow{\pi_Y^Z}(Y,\nu_Y)$ are $G$-factor maps with $\pi_Z\circ\pi_Y^Z=\pi_Y$. If, either
\begin{itemize}
    \item [(i)] $(B_0,\nu_0)$ is $\mu$-unique stationary, or
    \item [(ii)] $(B,\nu_B)$ is $\mu$-unique stationary,
\end{itemize}
then $(\tilde{Y},\nu_{\tilde{Y}})\overset{\pi_Z}{\cong}(Z,\nu_Z)$.
\end{corollary}
\begin{proof}
Since $\pi_{B_0}: (\tilde{Y},\nu_{\tilde{Y}})\to(B_0,\nu_0)$ is relatively measure-preserving, we have $h_\mu(\tilde{Y},\nu_{\tilde{Y}})=h_\mu(B_0,\nu_0)$. Since factor maps reduce Furstenberg entropy, we also have
$$h_\mu(B,\nu_B)\geq h_\mu(B_0,\nu_0)=h_\mu(\tilde{Y},\nu_{\tilde{Y}})\geq h_\mu(Z,\nu_Z).$$
Since $(Z,\nu_Z)$ is weakly amenable, by Corollary \ref{us leq am for G}, condition (i) will imply $h_\mu(Z,\nu_Z)\geq h_\mu(B_0,\nu_0)$, and (ii) will imply $h_\mu(Z,\nu_Z)\geq h_\mu(B,\nu_B)$. Hence we will always have
$$h_\mu(B_0,\nu_0)=h_\mu(\tilde{Y},\nu_{\tilde{Y}})=h_\mu(Z,\nu_Z).$$
By Theorem \ref{Thm B}, there exists a $G$-isomorphism $\beta_1: (B_0,\nu_0)\cong(\tilde{Y}_\RN,\nu_{\tilde{Y}_\RN})$ (note that $(B_0,\nu_0)$ is RN-irreducible as a $\mu$-boundary) with $\beta_1\circ\pi_{B_0}=\pi_{\tilde{Y}_\RN}$. There also exists a $G$-isomorphism $\beta_2:(Z_\RN,\nu_{Z_\RN})\cong (\tilde{Y}_\RN,\nu_{\tilde{Y}_\RN})$ with $\beta_2\circ\pi_{Z_\RN}\circ\pi_Z=\pi_{\tilde{Y}_\RN}$. Hence we have the following commutative diagram:
$$\begin{tikzcd}
                                                      & {(\tilde{Y},\nu_{\tilde{Y}})} \arrow[d, "\pi_{\tilde{Y}_\RN}"] \arrow[r, "\pi_Z"] \arrow[ld, "\pi_{B_0}"'] & {(Z,\nu_Z)} \arrow[d, "\pi_{Z_\RN}"]                          \\
{(B_0,\nu_0)} \arrow[r, "\beta_1"] \arrow[r, "\sim"'] & {(\tilde{Y}_\RN,\nu_{\tilde{Y}_\RN})}                                                                      & {(Z_\RN,\nu_{Z_\RN}).} \arrow[l, "\beta_2"'] \arrow[l, "\sim"]
\end{tikzcd}$$
Define the factor map $\pi_{B_0}^Z=\beta_1^{-1}\circ\beta_2\circ\pi_{Z_\RN}:(Z,\nu_Z)\to(B_0,\nu_0)$. Then $\pi_{B_0}^Z\circ\pi_Z=\pi_{B_0}$. Now we have that $(Z,\nu_Z)$ is an intermediate factor for both $\pi_Y:(\tilde{Y},\nu_{\tilde{Y}})\to(Y,\nu_Y)$ and $\pi_{B_0}:(\tilde{Y},\nu_{\tilde{Y}})\to(B_0,\nu_0)$. However, since $(\tilde{Y},\nu_{\tilde{Y}})$ is exactly the $\mu$-joining of $(Y,\nu_Y)$ and $(B_0,\nu_0)$ (realized by the factor maps $\pi_Y$ and $\pi_{B_0}$), the only possible common intermediate factor is $(\tilde{Y},\nu_{\tilde{Y}})$ itself. Therefore, we must have $(\tilde{Y},\nu_{\tilde{Y}})\overset{\pi_Z}{\cong}(Z,\nu_Z)$.
\end{proof}

The following corollary is \cite[Corollary 4.17]{HK24}
\begin{corollary}\label{HK24 Thm B}
Assume that $G$ is discrete, $(B,\nu_B)$ is $\mu$-unique stationary and $h_\mu(B,\nu_B)<+\wx$. Let $G\car(Z,\nu_Z)$ be a pmp action. Then the von Neumann algebra $G\ltimes L^\wx(B\times Z,\nu_B\times\nu_Z)$ is maximal amenable in $G\ltimes (B(L^2(B,\nu_B))\otimes L^\wx(Z,\nu_Z))$.
\end{corollary}
\begin{proof}
Following the same first step in the proof of \cite[Corollary 4.17]{HK24}, by taking commutants in $B(L^2(G)\otimes L^2(B\times Z))$ \cite[Proposition
 V.7.14]{TakI}, this is equivalent to proving that for any amenable von Neumann algebra $\CA$ with 
$$G\ltimes L^\wx(Z)\subset \CA \subset G\ltimes L^\wx(B\times Z),$$
one must have $\CA=G\ltimes L^\wx(B\times Z)$.

Let $(M,\tau)=L(G)$, and $\varphi(T)=\sum_{g\in G}\mu(g)\lag T \mathds{1}_g,\mathds{1}_g\rag$ be the hyperstate associated with $\mu\in\prob(G)$. Then we have $(\CB_\varphi,\zeta)=G\ltimes L^\wx(B,\nu_B)$ \cite[Theorem 4.1]{Izu04} and $h_\varphi(\CB_\varphi,\zeta)=h_\mu(B,\nu_B)<+\wx$ \cite[Example 5.3]{DP20}. By Example \ref{example USB}, $(\CB_\varphi,\zeta)$ is also $\varphi$-unique stationary.

Let $(\CB,\varphi_\CB)=G\ltimes L^\wx(B\times Z,\nu_B\times\nu_Z)$ and $\varphi_\CA=\varphi_\CB|_\CA$. Since $(\CA,\varphi_\CA)$ is amenable and $(\CB_\varphi,\zeta)$ is $\varphi$-unique stationary, by Corollary \ref{us leq am for M}, we have
$$h_\varphi(\CB_\varphi,\zeta)\leq h_\varphi(\CA,\varphi_\CA) \leq h_\varphi (\CB,\varphi_\CB) \leq h_\varphi (\CB_\varphi,\zeta).$$
Therefore, we must have $h_\varphi(\CA,\varphi_\CA)=h_\varphi (\CB,\varphi_\CB)$. By \cite[Theorem 3.6]{Zh23}, we further have $\CA_\RN=\CB_\RN$. Moreover, by \cite[Example 3.5]{Zh23}, since $\CB=G\ltimes L^\wx(B\times Z)$, we have $\CB_\RN=G\ltimes L^\wx((B\times Z)_\RN)$. Since $G\car(Z,\nu_Z)$ is a pmp action, we have 
$$(B\times Z,\nu_B\times\nu_Z)_\RN=(B,\nu_B)_\RN=(B,\nu_B).$$
Therefore,
$$G\ltimes L^\wx(B)=\CB_\RN=\CA_\RN\subset\CA.$$
Note that we also have $G\ltimes L^\wx(Z)\subset \CA$. Hence
$$G\ltimes L^\wx(B\times Z)=\lag G\ltimes L^\wx(B),G\ltimes L^\wx(Z)\rag \subset \CA,$$
which finishes the proof.
\end{proof}

As a corollary of Corollary \ref{max entropy for G} and \ref{us leq am for G}, we give a new proof to \cite[Theorem 9.2]{NS13}.
\begin{corollary}\label{USB 3 eq for G}
Assume that $(B,\nu_B)$ is $\mu$-unique stationary and $h_\mu(B,\nu_B)<+\wx$. Then for a RN-irreducible $(G,\mu)$-space $(X,\nu_X)$, the following conditions are equivalent:
\begin{itemize}
    \item [(i)] $h_\mu(X,\nu_X)=h_\mu(B,\nu_B)$;
    \item [(ii)] $(X,\nu_X)\cong(B,\nu_B)$ as $(G,\mu)$-spaces;
    \item [(iii)] $G\car(X,\nu_X)$ is amenable;
    \item [(iv)] $G\car(X,\nu_X)$ is weakly amenable.
\end{itemize}
In particular, $(B,\nu_B)$ does not admit any non-trivial amenable $G$-factor.
\end{corollary}
\begin{proof}
(i)$\Rightarrow$(ii): Assume that $h_\mu(X,\nu_X)=h_\mu(B,\nu_B)$, then by Corollary \ref{max entropy for G}, we have $(X,\nu_X)=(X_\RN,\nu_{X_\RN})\cong(B,\nu_B)$.

(ii)$\Rightarrow$(iii): This is clear since a Poisson boundary action is always amenable.

(iii)$\Rightarrow$(iv): This is clear since amenability implies weak amenability.

(iv)$\Rightarrow$(i): Since $(X,\nu_X)$ is weakly amenable and $(B,\nu_B)$ is $\mu$-unique stationary, by Corollary \ref{us leq am for G}, we have $h_\mu(X,\nu_X)\geq h_\mu(B,\nu_B)$. But we also have $h_\mu(X,\nu_X)\leq h_\mu(B,\nu_B)$. Hence $h_\mu(X,\nu_X)= h_\mu(B,\nu_B)$.
\end{proof}

The following Corollary \ref{USB 3 eq for M} is the noncommutative analogue of Corollary \ref{us leq am for G} \cite[Theorem 9.2]{NS13}, which also generalizes and gives a new proof to \cite[Theorem A]{Hou24}.
\begin{corollary}\label{USB 3 eq for M}
Assume that $(\CB_\varphi,\zeta)$ is $\varphi$-unique stationary and $h_\varphi(\CB_\varphi,\zeta)<+\wx$. Assume that $(\CA,\varphi_\CA)$ is a RN-irreducible $\varphi$-stationary W$^*$-extension of $(M,\tau)$. Then the following conditions are equivalent: 
\begin{itemize}
    \item [(i)] $h_\varphi(\CA,\varphi_\CA)=h_\varphi(\CB_\varphi,\zeta)$;
    \item [(ii)] $(\CA,\varphi_\CA)\cong(\CB_\varphi,\zeta)$ as W$^*$-extensions;
    \item[(iii)] $(\CA,\varphi_\CA)$ is amenable;
    \item[(iv)] $(\CA,\varphi_\CA)$ is weakly amenable.
\end{itemize}
In particular, such an inclusion $M\subset \CB_\varphi$ does not admit any non-trivial amenable intermediate subalgebra. For example, the inclusion $L(\Gamma)\subset L(\Gamma\car B)$ for a unique stationary Poisson boundary action of discrete group $\Gamma$ \cite[Example 4.6]{Zh24}.
\end{corollary}
\begin{proof}
(i)$\Rightarrow$(ii): Assume that $h_\varphi(\CA,\varphi_\CA)=h_\varphi(\CB_\varphi,\zeta)$, then by Corollary \ref{max entropy for M}, we have $(\CA,\varphi_\CA)\cong(\CB_\varphi,\zeta)_\RN=(\CB_\varphi,\zeta)$.

(ii)$\Rightarrow$(iii): This is clear since a noncommutative Poisson boundary is always amenable \cite[Proposition 2.4]{DP20}.

(iii)$\Rightarrow$(iv): This is clear since amenability implies weak amenability.

(iv)$\Rightarrow$(i): Since $(\CA,\varphi_\CA)$ is weakly amenable and $(\CB_\varphi,\zeta)$ is $\varphi$-unique stationary, by Corollary \ref{us leq am for M}, we have $h_\varphi(\CA,\varphi_\CA)\geq h_\varphi(\CB_\varphi,\zeta)$. But we also have $h_\varphi(\CA,\varphi_\CA)\leq h_\varphi(\CB_\varphi,\zeta)$. Hence $h_\varphi(\CA,\varphi_\CA)= h_\varphi(\CB_\varphi,\zeta)$.
\end{proof}

As a corollary of Corollary \ref{USB 3 eq for M}, we give a new proof to \cite[Theorem A]{Hou24}, which was used to prove the rank corollary of Connes’ rigidity conjecture under specific conditions \cite[Theorem B]{Hou24}. 
\begin{corollary}\label{Thm of Hou24}
Assume that $(\CB_\varphi,\zeta)$ is $\varphi$-unique stationary and $h_\varphi(\CB_\varphi,\zeta)<+\wx$. Let $M\subset \CA$ be a W$^*$-inclusion with $\CA$ amenable. Assume that there exists a normal faithful $M$-bimodular ucp map $\CP:\CA\to\CB_\varphi$, then we have that 
$$\CP|_{\mathrm{mult}(\CP)}:\mathrm{mult}(\CP)\to\CB_\varphi$$ 
is a $\ast$-isomorphism. 

In particular, we can view $\CB_\varphi\cong\mathrm{mult}(\CP)\subset\CA$ as a von Neumann subalgebra, and $\CP:\CA\to\CB_\varphi$ as a normal faithful conditional expectation.
\end{corollary}
\begin{proof}
Let $\varphi_\CA=\zeta\circ\CP$, then $\varphi_\CA$ is a $\varphi$-stationary normal faithful hyperstate on $\CA$. Let $(\CA_\RN,\varphi_{\CA_\RN})$ be the Radon-Nikodym factor of $(\CA,\varphi_\CA)$. Then by the definition of Radon-Nikodym factors, there exists a state preserving conditional expectation $E:(\CA,\varphi_\CA)\to(\CA_\RN,\varphi_{\CA_\RN})$. Hence $(\CA_\RN,\varphi_{\CA_\RN})$ is amenable and RN-irreducible. Moreover, since $(\CB_\varphi,\zeta)$ is $\varphi$-unique stationary, by Corollary \ref{USB 3 eq for M}, there exists a state preserving $M$-bimodular $\ast$-isomorphism $\CP_0:(\CA_\RN,\varphi_{\CA_\RN})\to(\CB_\varphi,\zeta)$. But by the uniqueness of noncommutative Furstenberg boundary maps \cite[Theorem 3.2(1)]{Zh23b}, we must have that $\CP|_{\CA_\RN}=\CP_0$ is a $\ast$-isomorphism. And $\CP_0^{-1}\circ\CP:\CA\to\CA_\RN$ is a faithful conditional expectation onto subalgebra. Hence we must have $\mathrm{mult}(\CP)=\mathrm{mult}(\CP_0^{-1}\circ\CP)=\CA_\RN$ and $\CP|_{\mathrm{mult}(\CP)}:\mathrm{mult}(\CP)\to\CB_\varphi$ is a $\ast$-isomorphism. 
\end{proof}

\begin{remark}
In the original statement of \cite[Theorem A]{Hou24}, the inclusion $M\subset\CB_\varphi$ was taken as $L(\Gamma)\subset L(\Gamma\car (B,\nu_B))$ for a unique stationary Poisson boundary action of a discrete group $\Gamma$, without the finite entropy assumption. In fact, the condition $h_\varphi(\CB_\varphi,\zeta)<+\infty$ in Corollary \ref{Thm of Hou24} can be omitted. The original proof of \cite[Theorem A]{Hou24} can be perfectly generalized to the setting of $M\subset\CB_\varphi$.
\end{remark}

\section*{Acknowledgement}
The author would like to thank his supervisor, Professor Cyril Houdayer, for numerous insightful discussions and valuable comments on this paper. The author is also grateful to Professor Anna Erschler, Professor Yair Hartman, and Professor Amos Nevo for their valuable comments from the perspective of ergodic group theory.

\end{document}